\documentclass{amsart}
\usepackage{graphicx}
\usepackage{amssymb,amscd,amsthm,amsxtra}
\usepackage{hyperref}
\usepackage{latexsym}
\usepackage{epsfig}
\usepackage{mathtools}
\usepackage{esint}
\usepackage{color}

\newtheorem{thm}{Theorem}[section]

\newtheorem{lem}[thm]{Lemma}

\theoremstyle{definition}
\newtheorem{defn}[thm]{Definition}
\theoremstyle{remark}
\newtheorem{rem}[thm]{Remark}
\numberwithin{equation}{section}
\newcommand{\R}{\mathbb R}

\newcommand{\p}{\partial}

\newcommand{\comment}[1]{}
\begin{document}
\title[Boundary Harnack Principle]{On the Parabolic Boundary Harnack Principle}
\author{D. De Silva}
\address{Department of Mathematics, Barnard College, Columbia University, New York, NY 10027}
\email{\tt  desilva@math.columbia.edu}
\author{O. Savin}
\address{Department of Mathematics, Columbia University, New York, NY 10027}\email{\tt  savin@math.columbia.edu}
\begin{abstract}We investigate the parabolic Boundary Harnack Principle by the analytical methods developed in \cite{DS1,DS2}. Besides the classical case, we deal with less regular space-time domains, including slit domains.
 \end{abstract}

%\thanks{}
%\subjclass{}%
%\keywords{One-phase free boundary problem; Harnack Inequality}
\maketitle

\section{Introduction}

\subsection{Statement of main results.} In this paper we provide direct analytical proofs of the parabolic Boundary Harnack Inequality for both divergence and non-divergence type operators, in several different settings. Our strategy is based on our earlier works \cite{DS1, DS2} where the elliptic counterparts of these results were obtained. In order to state our theorems precisely, we introduce some notation.

We denote by $\Gamma \subset \R^{n+1}$ the graph of a continuous function $g(x',t)$ in the $x_n$ direction,  
$$\Gamma:=\{x_n=g(x',t) \}, \quad \quad (0,0) \in \Gamma,$$
while $\mathcal C_r$ denotes the cylinder of size $r$ on top of $\Gamma$ (in the $e_n$ direction) i.e.,  
 $$\mathcal C_r:=\{ (x',x_n,t)| x' \in B'_r, \quad  t \in (-r^2,r^2), \quad g(x',t)<x_n < g(x',t)+r\}.$$ As usual, $x'=(x_1, \ldots, x_{n-1})$, while $B'_r \subset \R^{n-1}$ is the ball of radius $r$ centered at the origin.

We consider solutions $u(x,t)$ to the parabolic equation
$$ u_t=Lu  \quad \text{in $\mathcal C_1,$}$$
where $Lu = tr (A(x) D^2u)$ or $L(u)=div(A(x) \nabla u)$, with $A$ satisfying, 
$$\lambda I \leq A \leq \Lambda I,\quad 0<\lambda \leq \Lambda <+\infty.$$

First, we recall the standard boundary Harnack inequalities for parabolic equations in Lipschitz domains. References to known literature will be provided in the next subsection. Here $g \in C_{x',t}^{\alpha,\beta}$ if 
$$ |g(x',t)-g(y',s)| \le C(|x'-y'|^\alpha + |t-s|^\beta),$$ and  $\overline E$, $\underline E$ are points interior to $\mathcal C_1$ at times $t=1/2$ and $t=-1/2$ respectively, 
$$\overline E = \left( \left(g\left(0, \frac 12\right) + \frac 12 \right) e_n, \frac 12 \right ) , \quad \quad \underline E = \left( \left(g\left(0, -\frac 12\right) + \frac 12 \right) e_n, -\frac 12 \right ).$$

\begin{thm}[$C^{1, \frac 12}$ domains]\label{T1}
Assume that $g \in C_{x',t}^{1, \frac 1 2}$ and $u, v$ are two positive solutions to 
$$u_t=Lu, \quad \quad v_t=Lv \quad \mbox{in $\mathcal C_1$,} $$ 
with $u$ vanishing continuously on $\Gamma$. Then
\begin{equation}\label{PBH}
\frac{u}{v} \ (x) \le C \, \,  \frac {u(\overline E)} {v(\underline E) } \quad \mbox{for all  $x \in \mathcal C_{1/2}$},
\end{equation}
with $C$ depending only on $n$, $\|g\|_{C^{1,1/2}}, \lambda,$ and $\Lambda$. 
\end{thm}
\begin{figure}[h] 
\includegraphics[width=0.5 \textwidth]{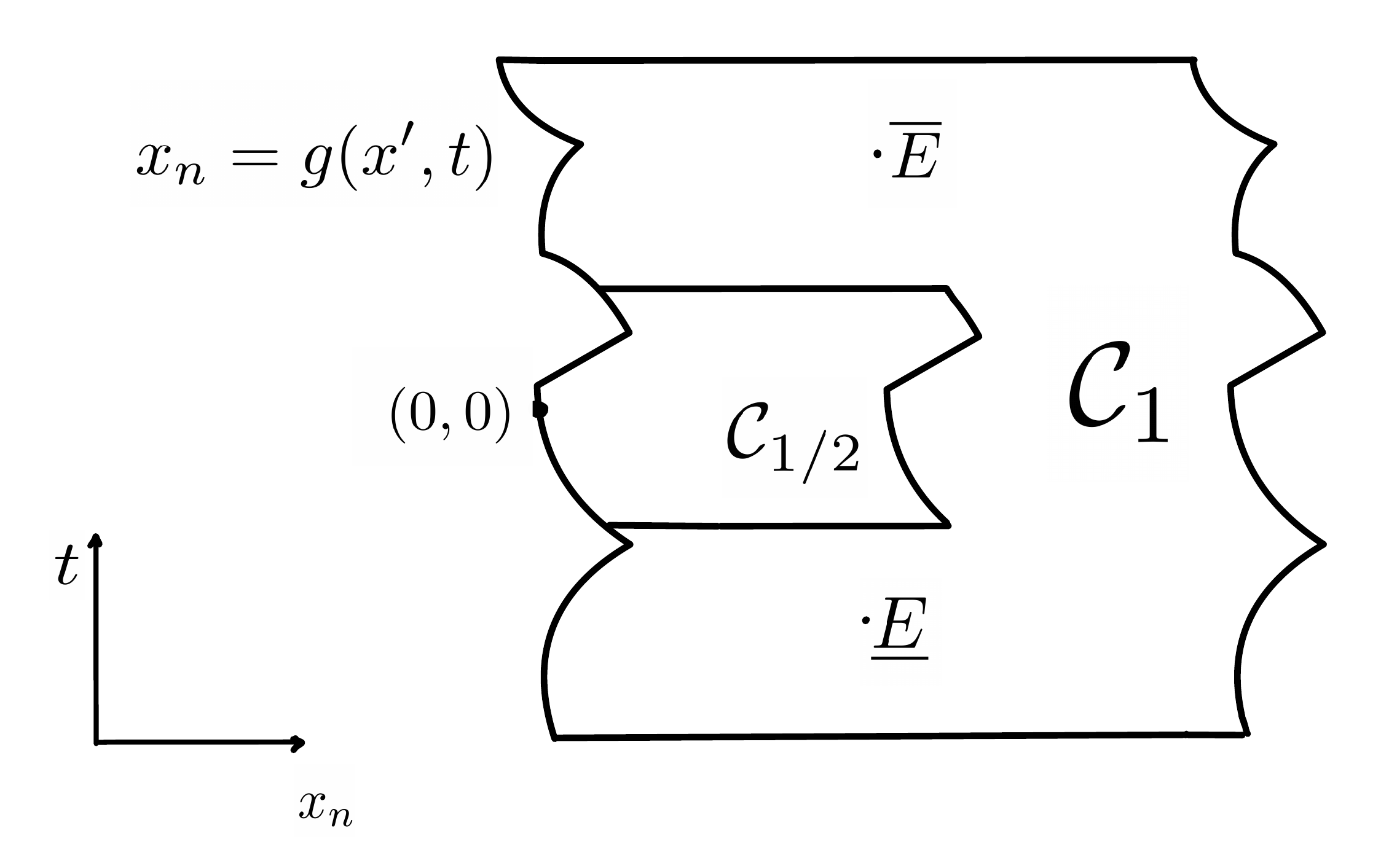}
\caption{Theorem \ref{T1}.} 
   \label{fig1}
\end{figure}

In this note, we provide new versions of Theorem $\ref{T1}$ in more general H\"older domains. 

\begin{thm}[$C^{\frac 12+, \frac 13+}$ domains] \label{T3}Theorem $\ref{T1}$ holds if $g \in C_{x',t}^{\alpha,\beta}$ with $\alpha > 1/2$, $\beta > 1/3$.
\end{thm}
In the case of the heat operator we may lower further the space regularity of $g$ to any exponent $\alpha >0$ provided that we have a  
$1/2$ H\"older modulus of continuity in time (from one-side). 
\begin{thm}[$C^{\alpha, \frac 12}$ domains] \label{T4}Theorem $\ref{T1}$ holds for the heat equation if $g \in C_{x',t}^{\alpha, \frac 12}$ with $\alpha > 0$. 
\end{thm}

We remark that the only property of the heat 
equation needed in the proof of Theorem \ref{T4} is the translation invariance with respect to the $x_n,t$ variables. Hence the theorem holds also for operators $L$ with coefficients depending only on the $x'$ variable.

Next we state a result in slit-domains, that is the case when the equations are satisfied in the complement of a thin set $S \subset \R^{n+1}$ included in a lower dimensional 
subspace. This case is relevant for example in the time-dependent Signorini problem. 

Precisely we assume that $S$ is a closed set and 
$$S \subset \{x_n=0\},$$ and in this case 
$$\mathcal C_r = B_r \times (-r^2,r^2),  \quad  \overline E=(1/2 e_n, 1/2), \, \, \underline E=(1/2 e_n,-1/2).$$ With these notation, we state our theorem.
\begin{thm}[Thin Parabolic Boundary Harnack]\label{T2} If $u, v$ are two positive solutions even in the $x_n$ variable, 
$$u_t=Lu \quad \quad v_t=Lv, \quad \mbox{in $\mathcal C_1 \setminus S$,} $$ 
and $u$ vanishes on $S$, then \eqref{PBH} holds. 
\end{thm}

The assumption that $u$, $v$ are even in the $x_n$ variable can be removed provided that  $\mathcal C_1 \setminus S$ contains a ball of radius $\sigma$ centered on $\{x_n=0\}$, and the constant $C$ in estimate \eqref{PBH} depends on $\sigma$. 

We remark that in Theorems \ref{T3}-\ref{T2} whenever the boundary of the domain contains non-regular points for the Dirichlet problem, the statement that $u$ vanishes on it is interpreted in the sense that $u$ is the limit of a sequence of continuous subsolutions which vanish on it.

\subsection{Known literature.} For the last 50 years, the boundary Harnack principle has played an essential role in analysis and PDEs in a variety of contexts. The available literature on this topic is very rich and we collect here only the crucial results, making no attempt to discuss the countless important applications of this fundamental tool.

\subsubsection{Elliptic case} In the elliptic context, the classical Boundary Harnack Principle, that is the case when $g$ is Lipschitz continuous, states the following. Here the notation is the same as above, with $u,v, g$ independent on $t$.
\begin{thm}\label{main} Let $u,v>0$ satisfy $ L u=L v=0$ in $\mathcal C_1$ and vanish continuously on $\Gamma$. Assume $u,v$ are normalized so that $u\left( e_n/2\right)=v(e_n/2)=1,$ then
\begin{equation}\label{BHI} C^{-1} \leq \frac u v \leq C, \quad \text{in $\mathcal C_{1/2},$}\end{equation} with $C$ depending on $n, \lambda, \Lambda,$
and the norm of $g$. \end{thm}

The case when $L= \Delta$ first appears in \cite{A,D,K,W}. Operators in divergence form were then considered in \cite{CFMS}, while the case of operator in non-divergence form was treated in \cite{FGMS}. The same result for operators in divergence form was extended also to so-called NTA domains in \cite{JK}. The case of H\"older domains and $ L$ in divergence form was  addressed with probabilistic techniques in \cite{BB1,BBB}, and an analytic proof was then provided in \cite{F}. For H\"older domains and operators $ L$ in non-divergence form, it is necessary that the domain is $C^{0,\alpha}$ with $\alpha>1/2$ or that it satisfies a uniform density property, and this was first established again using a probabilistic approach \cite{BB2}.

In \cite{DS1, DS2} we presented a unified analytic proof the Boundary Harnack Principle that does not make use of the Green's function and which holds for both operators in non-divergence and in divergence form. The idea is to find an ``almost positivity property" of a solution, which can be iterated from scale 1 to all smaller scales (some similar ideas were also used in \cite{KiS, S} to treat non-divergence equations with unbounded drift). This strategy successfully applies to other similar situations like that of H\"older domains, NTA domains, and to the case of slit domains, providing a unified approach to a large class of results.

\subsubsection{Parabolic case.} For parabolic equations the situation is more complicated, essentially due to the evolution nature of the latter which is reflected in a time-lag in the Harnack Principle. 
For operators in divergence form, the parabolic boundary Harnack principle in Theorem \ref{T1} is due to \cite{K2,FGS,Sal}. In the case of operators in non-divergence form in cylinders with $C^2$ cross sections, Theorem \ref{T1} was settled in \cite{G}, where the author also derived a Carleson estimate (see Lemma \ref{CE}) in Lipschitz domains. The statement of Theorem \ref{T1} in Lipschitz domain was later obtained in
 \cite{FSY}, which is (to the authors knowledge) the first instance in which a boundary Harnack type result in Lipschitz domains is obtained without the aid of Green's functions (and it is probably the inspiration for the later works in the elliptic context \cite{KiS, S}). In \cite{HLN}, Theorem \ref{T1} was also shown to hold for unbounded parabolically Reifenberg flat domains.  In the context of time independent H\"older domains, a result  in the spirit of Theorem \ref{T3} was obtained via probabilistic techniques in \cite{BB}. The result in Theorem \ref{T4} is completely novel. Concerning slit domains, in the case when $S$ is the subgraph of a parabolic Lipschitz graph, the thin-version Theorem \ref{T2} was established by \cite{PS}. Again, our strategy provides a unified approach for a variety of contexts.

\subsection{Organization of the paper.} The paper is organized as follows. In Section 2, after recalling some standard results, we provide the proof of Theorems \ref{T1} and \ref{T2}. The key ``almost positivity" property to be iterated from scale 1 to all smaller scales, is obtained in Lemma \ref{ML}. The following section deals with H\"older domains and the proof of Theorem \ref{T3}, which relies on the same strategy as Theorem \ref{T1}, though the proof of the Carleson estimate in the H\"older setting requires a more involved argument similar to the one in the proof of Lemma \ref{ML}. Section 4 contains the proof of Theorem \ref{T4}, which is based on refined versions of the weak Harnack inequality (see Lemmas \ref{l1}-\ref{l2}).

\section{Proof of Theorem \ref{T1} and \ref{T2}}

In this section, we provide the proof of the classical result Theorem \ref{T1} and the novel result Theorem \ref{T2}. We start by collecting standard known Harnack type inequalities. In the divergence setting these results are due to \cite{M}, while in the non-divergence setting they follow from \cite{Wa}.

\subsection{Weak Harnack inequality}

Denote by $$Q_r:=(-r,r)^n \times (-r^2,0], \quad \quad Q_r(x_0,t_0) := (x_0,t_0) + Q_r,$$ the parabolic cubes of size $r$.
The parabolic boundary of $Q_r$ is denoted by $\p_p Q_r$ and is given by:
$$\p_p Q_r:= (\p (-r,r)^n \times (-r^2,0)) \cup ((-r,r)^n \times \{-r^2\}).$$

Similarly, 
$$Q'_r:=(-r,r)^{n-1} \times (-r^2,0], \quad \quad Q'_r(x'_0,t_0) := (x'_0,t_0) + Q'_r.$$ 
Our main tools in establishing the boundary Harnack inequalities are the standard weak Harnack estimates. We recall the parabolic versions which as mentioned in the introduction
differ from the elliptic counterparts due to the time-lag.   

\begin{thm}[Supersolution]
If $$ u_t \ge Lu \quad \mbox{and} \quad u \ge 0 \quad \quad \mbox{in $Q_1$}, \quad u(0,0)=1,$$
then
$$ \int_{Q_{\frac 12}(0,-\frac 12)} u^p \, dxdt \le C,$$
for some $p>0$ small, $C$ large universal (i.e. dependent on $n,\lambda, \Lambda$).
\end{thm}

\begin{thm}[Subsolution]\label{whi2}
If $$ u_t \le Lu \quad \mbox{and} \quad u \ge 0 \quad \quad \mbox{in $Q_1$},$$
then 
$$ u(0,0) \le C(p) \| u\|_{L^p(Q_1)},$$
for any $p>0$.
\end{thm}

The classical (backward) Harnack inequality then reads as follows.

\begin{thm}[Harnack inequality]\label{Har}
If $$ u_t = Lu \quad \mbox{and} \quad u \ge 0 \quad \quad \mbox{in $Q_1$},$$
then for $c$ small universal (dependent on $n,\lambda,\Lambda$),
$$ \min_{Q_{1/2}}u \ge c \, \, \max_{Q_{1/2}(0,- \frac 1 2)}u.$$
\end{thm}

Another useful version for the subsolution property is the following measure to pointwise estimate.

\begin{thm}[Subsolution]\label{whi4}
If $$ u_t \le Lu \quad \mbox{and} \quad 1 \ge u \ge 0 \quad \quad \mbox{in $Q_2$},$$
and for some $\delta>0,$ $$|\{u=0\} \cap Q_1(0,-1)| \ge \delta.$$ Then $$ u \le 1-c(\delta) \quad \mbox{in} \quad Q_{1}.$$

\end{thm}

With these tools at hand, we are ready to provide in the following subsection our proof of the classical result in Theorem \ref{T1}.
\subsection{Proof of Theorem \ref{T1}}
In what follows, constants depending on $n, \lambda, \Lambda$, and the norm of $g$, are called universal.

We denote by
$$\mathcal C_r^-:=\{ (x',x_n,t)| \, \, x'\in (-r,r)^{n-1}, \quad  t \in (-r^2,0], \quad g(x',t)<x_n < g(x',t)+r\},$$
the backward in time cylinder of size $r$ on top (in the $e_n$ direction) of the graph $\Gamma$ of $g$. Also we set, 
$$\mathcal A_{r}:=\left\{(x,t) \in \mathcal C_r^-| \quad g(x',t)+ \delta r \le x_n < g(x',t) + r \right \}, $$
that is the collection of points in the cylinder $\mathcal C^-_r$ at height greater or equal than $ \delta r$ on top of $\Gamma$, for some $\delta>0$ small, to be made precise later. 
\begin{figure}[h] 
\includegraphics[width=0.5 \textwidth]{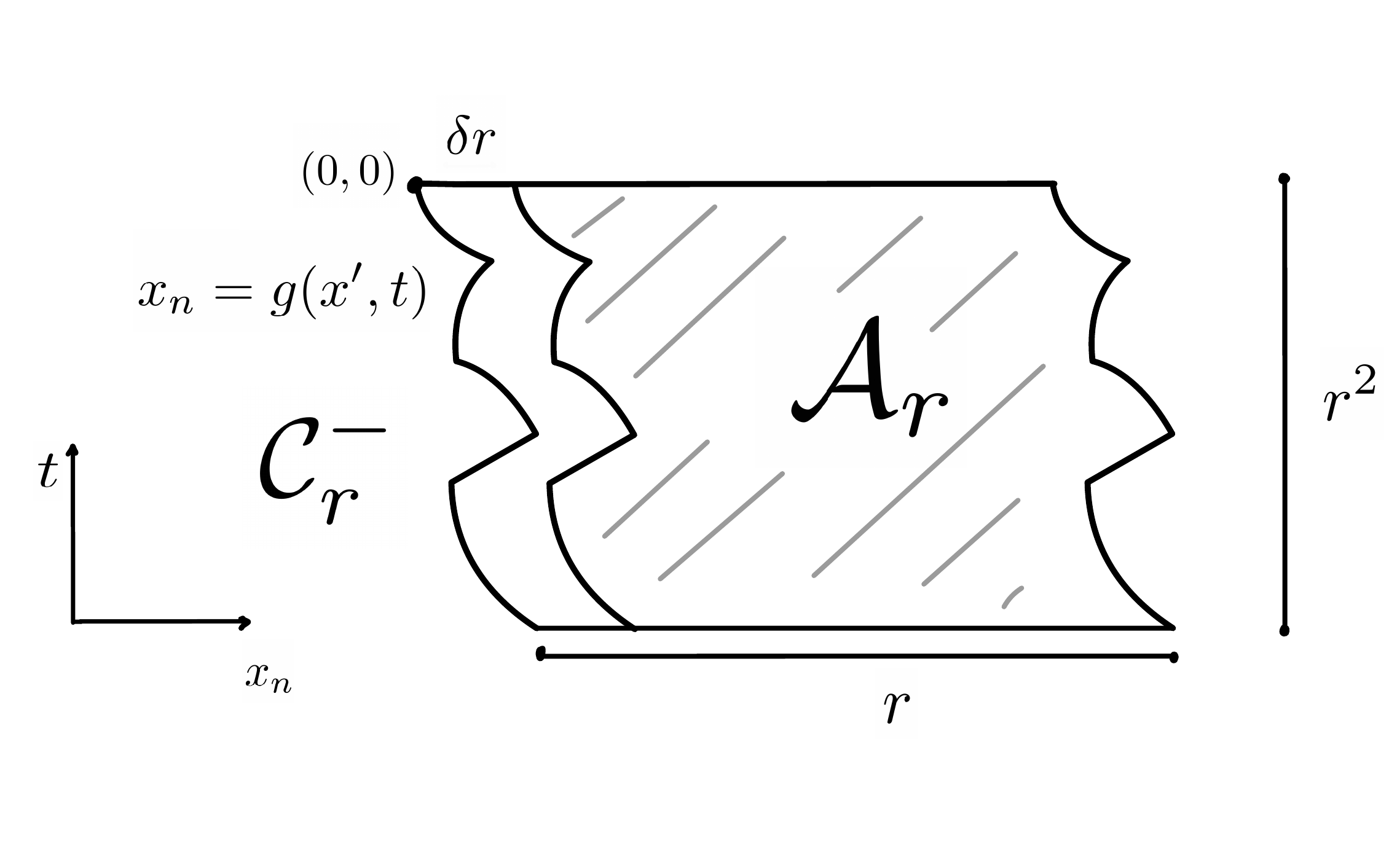}
\caption{ The sets $\mathcal C_r^-$ and $\mathcal A_{r}$.}
   \label{fig2}
\end{figure}

The key tool for establishing the boundary Harnack estimates is the following iterative lemma. Later we will apply this lemma for the difference $w= v - c u$ for some sufficiently small constant $c$, in order to obtain the desired claim in Theorem \ref{T1}. 
\begin{lem}\label{ML}
There exist universal constants $M$, $\delta>0$, such that if $w$ is a solution to $$ w_t=Lw \quad \mbox{ in} \quad \mathcal C^-_r,$$ (possibly changing sign) with  $w^-$ vanishing continuously on $\Gamma$, 
\begin{equation}\label{a1} w \geq M \quad \quad \text{in} \quad \mathcal A_{r}, \end{equation}
and
$$w \geq -1 \quad \text{in $\mathcal C^-_r$},$$
then,
\begin{equation}\label{301}
w \geq M \, a \quad \text{in $\mathcal A_{\frac r 2}$},
\end{equation} and
\begin{equation}\label{302}
w \geq -a \quad \text{in $\mathcal C^-_{\frac r 2}$}, 
\end{equation}
for some small $a>0$.
\end{lem}

The conclusion can be iterated indefinitely and we obtain that if the hypotheses are satisfied in $\mathcal C^-_{r}$ then 
\begin{equation}\label{pos}w > 0 \quad \text{on the line segment $\{(se_n,0), \quad 0<s<r\}.$}\end{equation}

\begin{proof} We start by observing that 
any point $(x_0,t_0) \in \mathcal A_{r/2}$ can be connected through a chain of backwards-in-time adjacent parabolic cubes of size $\bar r:=c_1 \delta r$ centered at $$ (x_j,t_j):=(x_0 + j \bar r e_n, t_0 - j \bar r ^2),$$ 
to a last cube $Q_{\bar r}(x_m,t_m) \subset \mathcal A_r$ (see Figure 2).
Here $c_1$ is small depending on the $C_{x',t}^{1,1/2}$ norm of $g$ so that $$Q_{2 \bar r} (x_j,t_j) \subset \mathcal C_r^-,$$ and the number $m$ of cubes depends only on $c_1$. By Harnack inequality (Theorem \ref{Har}) applied to $w+1 \ge 0$, using assumption \eqref{a1}, we get
$$(w +1)(x_0,t_0) \ge c^m (M+1)  \quad \Longrightarrow \quad w \ge 1 \quad \mbox{in} \quad \mathcal A_{r/2},$$
provided that we choose $M$ large depending on $c_1$ (and independent of $\delta$). Hence \eqref{301} holds with $a= 1/M$. 

To establish \eqref{302} with this choice of $M, a$, we first extend $w^-=0$ in 
$$Q'_r \times (\{x_n < g(x',t)\} \cup \{x_n > g(x',t) + \delta r\}),$$ so that  $w^-$ is a global subsolution in $Q'_r \times \R$ thanks to assumption \eqref{a1}.Then, for each cube $Q_{2\delta r}(x,t)$ satisfying $Q'_{2\delta r} (x', t) \subset Q'_r,$ we 
have
$$ |\{w^-=0\} \cap Q_{2\delta r}(x,t)| \ge \frac 12 |Q_{2\delta r}(x,t)|.$$
This is a consequence of the graph property of $\Gamma$. Indeed, for each fixed $(x',t)$, we consider the 1D line in the $e_n$ direction. Any segment 
of length $2 \delta r$ on this line has at least half of its length either in $\mathcal A_r$ or in the complement of $\mathcal C_r^-$.

By weak Harnack inequality, Theorem \ref{whi4}, as we remove the collection of cubes $Q_{2\delta r}(x,t)$ which are tangent to the parabolic 
boundary of $Q'_r \times \R$, the norm $\|
w^-\|_{L^\infty}$ decays by a factor $1-c$, $c>0$ universal. Iterating this for $\sim 1/\delta$ times we find that
$$w^- \le (1-c)^{1/\delta} \quad \mbox{in $\mathcal C^-_{r/2} \subset Q'_{r/2} \times \R$.}$$ We choose $
\delta$ small, so that $w^- \le a=M^{-1}$ and \eqref{302} holds. 

\end{proof}

\begin{figure}[h] 
\includegraphics[width=0.5 \textwidth]{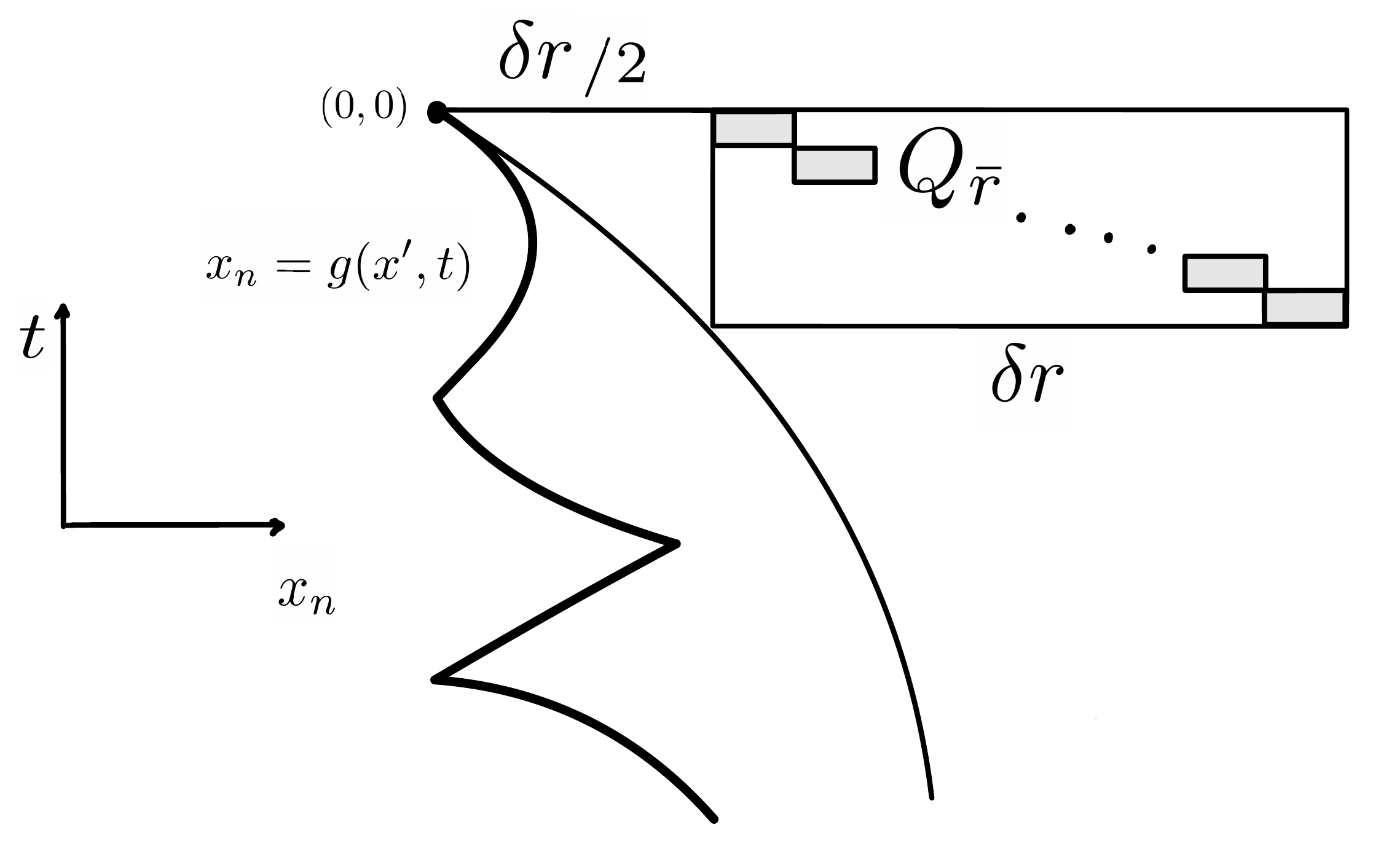}
\caption{Proof of Lemma \ref{ML}.} 
   \label{fig3}
\end{figure}

A second ingredient in the proof of Theorem \ref{T1} is the following Carleson estimate which provides a bound for $u$  in the cylinder $\mathcal C_{2/3}$.

\begin{lem}[Carleson estimate]\label{CE} Let $u, \bar E$ be as in Theorem $\ref{T1}$, then 
$$\|u\|_{L^\infty (\mathcal C_{2/3})} \le C \, u(\overline E),$$ with $C>0$
 universal.\end{lem} 

\begin{proof}
The Carleson estimate can be established by similar arguments as in the Lemma \ref{ML} above. 
We will use this approach in the case of H\"older domains in the next section. However, for $C_{x',t}^{1,\frac 12}$ domains, 
the Carleson estimate is a direct consequence of the weak Harnack inequality.

Indeed, assume that $u(\overline E)=1$. Any point $(x_0,t_0) \in \mathcal C_{12/17}$ can be connected to $\overline E$ by a chain of forward-in-time adjacent cubes $Q_{r_j}(x_j,t_j)$ included in $\mathcal C_1$, with $r_j$ proportional to the parabolic distance $d_j$ from $(x_j,t_j)$ to $\Gamma$. The number of cubes in this chain is proportional to $|\log d_0|$. By Harnack inequality,
$$u(x_0,t_0) \le \, e^{C |\log d_0|} u(\overline E) \,  \le d_0^{-C'}.$$
This means that $\|u\|_{L^p} \le C$ in $\mathcal C_{12/17}$ for some small $p>0$ universal. The extension of $u$ by $0$ in $\Omega_r:=((-r,r)^{n-1} \times (-r^2,r^2)) \times \{x_n \leq g(x',t)\}$ is a subsolution, and now we can apply weak Harnack inequality Theorem \ref{whi2} in cubes $Q_{c_0}(x,t) \subset \mathcal C_{12/17} \cup \Omega_r$ for $(x,t) \in \mathcal C_{2/3}$ and $c_0$ small universal, to obtain the desired conclusion.

\end{proof}

We are now ready to combine the previous two lemmas and obtain the desired Theorem \ref{T1}.

\begin{proof}[Proof of Theorem $\ref{T1}$]
We assume that $u(\bar E)=v (\underline E)=1$ and define $w=C_1 v - c_1 u$. By Harnack inequality applied to $v$ and the Carleson estimate for $u$, 
we can choose the constants $C_1$ large, $c_1$ small (depending on $\delta$, $M$) such that $w$ satisfies$$w \ge -1 \quad \text{in $\mathcal C_{2/3},$} \quad \mbox{and} \quad w(x,t) \ge M \quad \mbox{if} \quad x_n \geq g(x',t) + \delta/4.$$
 Then we can apply Lemma \ref{ML} in cylinders $\mathcal C^-_{1/6}$ around any point on $\Gamma \cap \mathcal C_{1/2}$,
and conclude from \eqref{pos} that $ w >0$ in $\mathcal C_{1/2}$.
\end{proof}

\subsection{Proof of Theorem \ref{T2}}
The proof is identical to the one of Theorem \ref{T1} after the appropriate modifications in the definitions of $\mathcal C_r^-$ and $\mathcal A_r$. Precisely, 
$$\mathcal C_r^-:=Q_r \setminus S, \quad \quad \mathcal A_r:= Q_r \cap \{|x_n| \ge \delta r\}.$$
Lemma \ref{ML} applies for the difference $w=v-cu$. 
The hypotheses that $u$ and $w^-=(cu-v)^+$ vanish on $S$ are understood in the sense that 
each of them is obtained in $\mathcal C_1^-$ as a pointwise limit of an increasing sequence of continuos subsolutions in $Q_1$ which vanish on $S$.  
Notice that if 
$u_n$ is such a sequence for $u$, then $(c \, u_n - v)^+$ is a corresponding sequence for $w^-$, (since $v \ge 0$ in $\mathcal C_1^-$). Thus, the 
extensions of $u$ and $w^-$ by $0$ on $S$ are subsolutions in $Q_1$, and Lemmas \ref{ML} and \ref{CE} hold as above.

\qed

\section{H\"older domains and the proof of Theorem \ref{T3}} 

In this section we prove Theorem \ref{T3} by extending the arguments of the previous section to H\"older domains. 
We assume that for some $\alpha >\frac 12$, 
\begin{equation}\label{4000}
[g]_{C_{x',t}^{\alpha,\frac{\alpha}{1+\alpha}}} \le K,
\end{equation} 
for some constant $K$. Below, constants depending possibly on $n, \lambda, \Lambda, \alpha$ and $K$ are called universal.

We define 
$$\mathcal C_r^-:=\{ (x',x_n,t)| \, \, x\in (-r,r)^{n}, \quad  t \in (-r,0], \quad g(x',t)<x_n < g(x',t)+r\},$$
and notice that here we took the time interval of $\mathcal C_r^-$ of size $r$ instead of the natural parabolic scaling $r^2$ that we used in the previous 
section. This change is due to the fact that the norm of $g$ is no longer left invariant by the parabolic scaling. We also define 
$$\mathcal A_{r}:=\left\{x \in \mathcal C_r^-| \quad g(x',t)+r^\beta \le x_n < g(x',t) + r \right \}, $$
the points in the cylinder $\mathcal C^-_r$ at height greater or equal than $r^\beta$ on top of $\Gamma$, for some $\beta>1$ to be made precise later.

\begin{lem}\label{ML2}
Suppose \eqref{4000} holds for $\mathcal C_r^-$ and let $w$ be a solution to $$ w_t=Lw \quad \mbox{ in} \quad \mathcal C^-_r,$$ for which $w^-$ vanishes on $\Gamma$. There exist universal constants $C_0, \beta>0$ such that if 
$$w \geq f(r) \quad \quad \text{on} \quad \mathcal A_{r}, $$
and
$$w \geq -1 \quad \text{on $\mathcal C^-_r$},$$
where $$f(r):= e^{C_0r^\gamma}, \quad \gamma:=\beta(1-\frac{1}{\alpha}) <0,$$
then,
\begin{equation}\label{401}
w \geq f( \frac r 2) \, \, a \quad \text{on $\mathcal A_{\frac r 2}$},
\end{equation} and
\begin{equation}\label{402}
w \geq -a \quad \text{on $\mathcal C^-_{\frac r 2}$}, 
\end{equation}
for some small $a=a(r)>0$, as long as $r\leq r_0$ universal.
\end{lem}

\begin{proof} We adapt the argument of Lemma \ref{ML} in this case and sketch the details. 

We connect a point $(x_0,t_0) \in \mathcal A_{r/2}$ (which is not in $\mathcal A_r$) to a point $(x_m,t_m)$ with $x_m= x_0 + r^\beta e_n \in \mathcal A_r$ by a chain of adjacent backward-in-time cubes of size $\bar r := c_0 \,  r^{\beta/\alpha}$. The number $m$ of cubes depends on $r$, i.e. 
$$m \sim r^{\beta}/ \bar r = c_0^{-1} r^{\beta (1-\frac 1 \alpha)}= c_0^{-1} r^\gamma.$$
All the cubes are included in the domain ($t_m:=t_0-m \bar r^2$)
$$\left \{ (x- x_0) \cdot e_n \ge 0, \quad t \in [t_m, t_0], \quad (x-x_0)' \in [-\bar r, \bar r]^{n-1} \right\},$$
which by \eqref{4000} is included in $\mathcal C_r^-$ since $m \bar r^2 \sim r^\beta \bar r = c_0 r^{\beta \frac{\alpha+1}{\alpha}}$, and $c_0$ is chosen 
small. Moreover, $Q_{\bar r}(x_m,t_m) \subset \mathcal A_r$, and Harnack inequality for $w+1$ implies that 
\begin{equation}\label{step1}
w+1 \ge f(r) e^{-C m} \ge 2 \quad \quad \mbox{in $\mathcal A_{r/2}$,} 
\end{equation} where the last inequality is guaranteed if we choose $C_0$ sufficiently large.

 For the second step which bounds $w^-$ we use cylinders of size $2r^\beta$ (instead of $2 \delta r$ as before) and get by the same argument as in the Lipschitz case
 \begin{equation}\label{step2}
 w^- \le e^{-c r^{1-\beta}}=:a.
 \end{equation}
 The conclusion follows since in $\mathcal A_{r/2}$, $w \ge 1 \ge f(r/2) a$, and in the last inequality we used $1-\beta < \gamma$, provided that $\beta$ 
 is chosen sufficiently large.

\end{proof}

\begin{lem}[Carleson estimate] \label{CE2} Let $u, \bar E$ be as in Theorem $\ref{T3}$. Then, 
$$\|u\|_{L^\infty (\mathcal C_{1/2})} \le C \, u(\overline E),$$
with $C$ universal.
\end{lem}

\begin{proof}
We apply an iterative argument similar to the one of Lemma \ref{ML2} above.

Assume $u(\overline E)=1$, and denote by $h_\Gamma$ the distance in the $e_n$ direction between a point $(x,t) \in \mathcal C_1$ and $\Gamma$
$$h_\Gamma(x,t):=x_n-g(x',t).$$
Any point $(x,t) \in \mathcal C_{2/3}$ can be connected to $\overline E$ by a chain of adjacent forward-in-time cubes included in $\mathcal C_r$, so that the size of each cube is proportional to the distance from its center to $\Gamma$ raised to the power $1/\alpha$. The H\"older continuity of $g$ implies that the number of cubes in this chain is proportional to $ (h_\Gamma(x,t))^{1-1/\alpha}$, and by Harnack inequality we find
\begin{equation}\label{int}u \leq e^{\, C_1 h_\Gamma^{1-1/\alpha}} \quad \quad \text{in $\mathcal C_{2/3}$},
\end{equation}
 with $C_1$ universal.

With the same notation as in Lemma \ref{ML2}, we wish to prove that if $r \le r_0$ and
$$u(y,t_y) \geq f(r),$$ for some $(y,t_y) \in \mathcal C_{1/2}$,
then we can find $(z,t_z) \in \mathcal S$, 

$$\mathcal S:=\left\{(x,t)| \quad x'-y'\in (-r,r )^{n-1}, \quad t \in (t_y - r, t_y], \quad 0<h_{\Gamma}(x,t)< r^\beta\right \},$$ such that
$$u(z,t_z) \geq f\left(\frac{r}{2}\right).$$
Since $|(z,t_z)-(y,t_y)|\le C r^\alpha$, we see that for $r$ small enough, we can build a convergent sequence of points $(y_k,t_k) \in \mathcal C_{2/3}$ 
with $u(y_k,t_k)\ge f(2^{-k}r )\to \infty$. This is a contradiction if we assume that $u$ vanishes continuously on $\Gamma$, and is therefore bounded. 
If $u=0$ on $\Gamma$ is understood in the sense that $u$ is the limit of an increasing sequence of continuous subsolutions which vanish on $\Gamma$, then we may apply the argument below to one such subsolution and reach again a contradiction.

To show the existence of the point $z$, assume for simplicity $y'=0$, $t_y=0$ and then $\mathcal S= \mathcal C^-_r \setminus \mathcal A_r$. Let 
$$w:=\left(u- \frac 12 e^{C_0 r^\gamma}\right)^+, \quad \quad \mbox{with} \quad C_0 \gg C_1.$$ 
By \eqref{int} we know that 
$$w=0 \quad \mbox{ in $\mathcal A_r$}.$$ 
If our claim is not satisfied then we apply Weak Harnack inequality for $w$ in cubes of size $2 r^\beta$ repeatedly as in Lemma \ref{ML2}. As we move 
a distance $r$ inside the domain we obtain
\begin{equation}\label{step22}
w \leq f\left(\frac r2\right) \, \, e^{-c_0 r^{1-\beta}} \quad \text{in $\mathcal C^-_{r/2}$}.
\end{equation}
In particular $$\frac 12 f(r) \le w(y,t) \le f\left(\frac r2\right) e^{-c_0 r^{1-\beta}},$$
and we reach a contradiction if $r_0$ is sufficiently small as long as $1-\beta < \gamma$ (which is possible because $\alpha>1/2$).

\end{proof}

\section{Proof of Theorem \ref{T4}}

In this section we assume that \eqref{4000} holds for some $\alpha>0$ possibly small, and in addition $g$ satisfies a one-sided $C^{1/2}$ bound in the 
$t$ variable, i.e.
\begin{equation}\label{lower} g(x',t+s) - g'(x',t ) \ge - K s^{1/2}, \quad \quad \mbox {if $s \ge 0$}.\end{equation}

We will improve the estimates \eqref{step2}, \eqref{step22} of the previous section by applying weak Harnack inequality in parabolic cubes of smaller size 
$\bar r \sim r^{\beta/\alpha}$ (which is the size chosen in the first step to obtain \eqref{step1}) instead of $r^\beta$. Then the oscillation of $w^-$ (or $w$) will decay by a factor $e^{-c r^{1-\beta/\alpha}}$ as we go from $\mathcal C_r^-$ to $\mathcal C_{r/2}^-$.  
However, in cubes of size $\bar r$ we can no longer guarantee the uniform measure estimate of the set where $w^-=0$. To deal with this, we introduce a notion of parabolic 
capacity for the heat equation. This allows us to diminish the oscillation of $w^-$ more precisely than in the measure estimate of Theorem \ref{whi4}.

\begin{defn}\label{Cap}
Let $E$ be a closed set. 
Set, 
$$cap_{Q_1}(E):= \varphi (0,1) $$
where $\varphi $ is the solution to the heat equation in $Q_2(0,1) \setminus (E \cap \overline Q_1)$ which equals 0 on the parabolic boundary of $Q_2(0,1)$ and it is equal to 1 in $E \cap \overline Q_1$.
\end{defn}
The function $\varphi$ is well-defined by the Perron-Wiener-Brelot-Bauer theory (see for example \cite{Fr}). 
Similarly, we can define $cap_{Q_r(x,t)}(E)$ by translating the cube at the origin, and then performing a parabolic rescaling
$$ cap_{Q_r(x,t)}(E):=cap_{Q_1}( \tilde E), \quad \quad \tilde E:=\{(y,s)| (x+r^2 y, t+ rs) \in E \}.$$

We prove here two lemmas about weak Harnack inequality depending on the size of the capacity of $E$ in $Q_1$.
The first lemma states that a solution to the heat equation in $Q_1 \setminus E$ satisfies the Harnack inequality in measure if $E$ has small capacity. 

\begin{lem}\label{l1}
Assume $v \ge 0$ is defined in $Q_1 \setminus E$ and satisfies
$$ v_t=\triangle v.$$ 
Let $$Q^i:=Q_{1/4}(x_i,t_i) \subset Q_1, \quad i=1,2$$ be two cubes of size $1/4$ included in $Q_1$, with $t_2 -t_1 \ge 1/4$. Assume that 
$$cap_{Q_1}(E) \le \delta \quad \mbox{and} \quad \frac{|\{v \ge 1\} \cap Q^1|}{|Q^1|} \ge 1/2,$$
for some $\delta$ small universal. 
Then $$\frac{|\{v \ge c_0\} \cap Q^2|}{|Q^2|} \ge 1/2$$
for some $c_0$ small universal.
\end{lem}

\begin{proof}
Let $h$ be the solution to the heat equation in $Q_1 \setminus K$ with $h=0$ on the parabolic boundary of $Q_1$, and $h=1$ on $K:=\{v \ge 1\} \cap Q^1$. 
We claim that $$v \ge h-\varphi \quad \mbox{in $Q_1 \setminus E$,}$$
where $\varphi$ is the function from Definition \ref{Cap}. Since both $v$ with $h-\varphi$ solve the heat equation in $Q_1 \setminus (K \cup E)$, it 
suffices to check the claim on the parabolic boundary of $Q_1 \setminus E$ and on $K$.

Indeed, $v \ge 0 \ge h -\varphi$ on $\partial_p Q_1$, and $v\ge 1 \ge h-\varphi$ on $K$. Moreover, $h \le 1 \le \varphi$ on $E$ gives $h-\varphi \le 0 $ on $E$, and since $v \ge 0$ the claim is proved. 

The conclusion follows from the inequality above, since by the Weak Harnack inequality, there exists $c_0$ small universal such that
$h \ge 2 c_0$ in $Q^2$. On the other hand, $\varphi(0,1)= cap_{Q_1}E \le \delta$ implies that $\varphi \le c_0$ in half the measure of $Q^2$ provided that $\delta$ is chosen sufficiently small. \begin{figure}[h] 
\includegraphics[width=0.5 \textwidth]{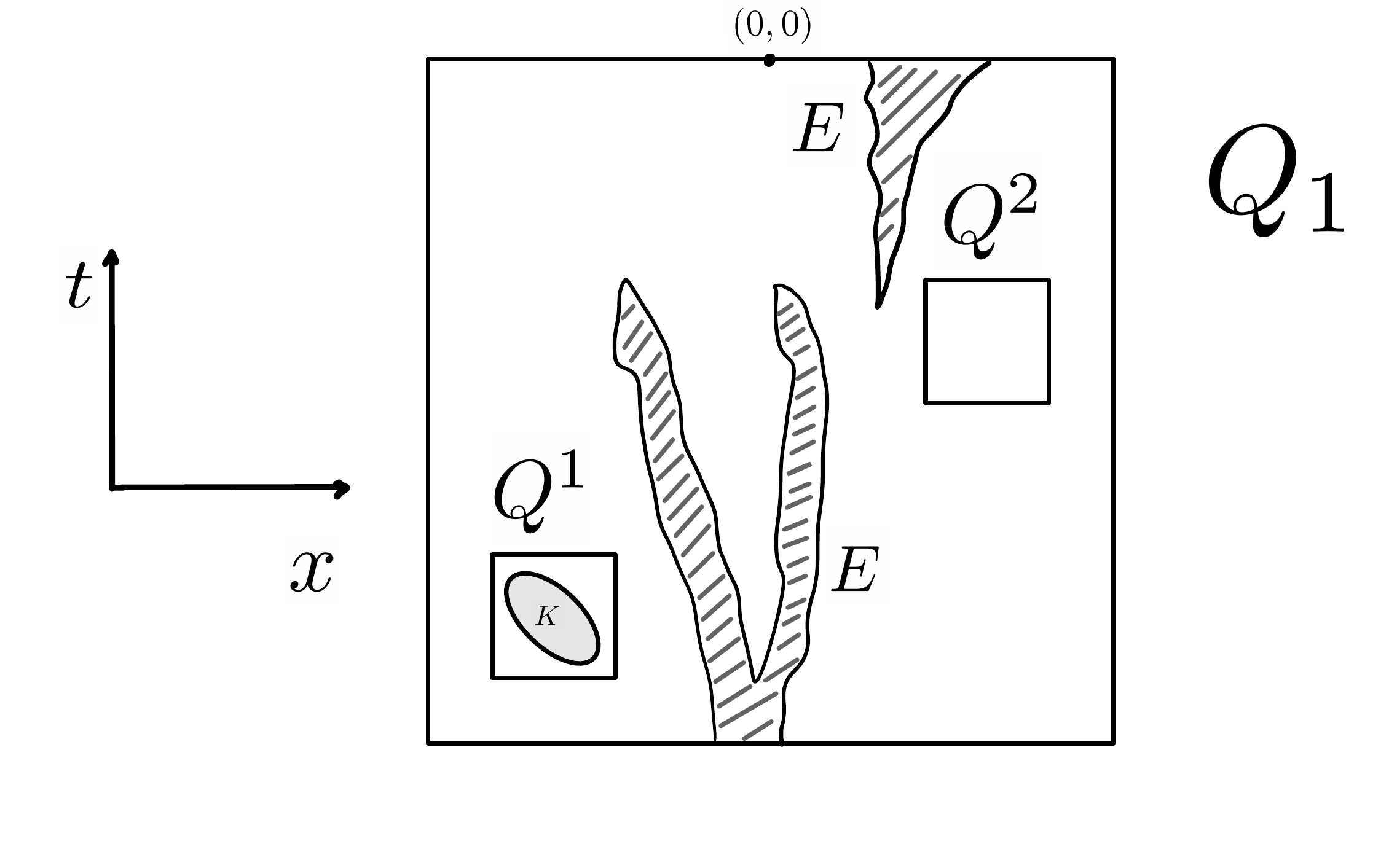}
\caption{Lemma \ref{l1}. }
   \label{fig4}
\end{figure}

\end{proof}

\begin{rem}\label{R40} We may use cubes $Q^i$ of size $\sigma$ and with $t_2 - t_1 \ge \sigma^2$, as long as $\delta$ and $c_0$ are allowed to depend on $
\sigma$ as well.
\end{rem}

The second lemma states that the weak Harnack inequality holds for a subsolution $v\ge 0$ which vanishes on a set $E$ of positive capacity. It follows directly from the definition of $cap_{Q_1}(E)$.
\begin{lem}\label{l2}
Assume that $v \ge 0$ in $Q_2(0,1)$, and 
$$\mbox{$\triangle v \ge v_t$ in $Q_2(0,1)$, and $v=0$ in $E \cap \overline Q_{1}$.}$$ If for some $\delta>0$
$$ cap_{Q_1}(E) \ge \delta,$$ then
$$v(x,t) \le (1-c(\delta)) \|v\|_{L^\infty}, \quad (x,t)\in Q_{1/2}(0,1).$$

\end{lem}

\begin{proof}
Assume $\|v \|_{L^\infty}=1$. We compare $1-v$ with $\varphi$ in $Q_2(0,1) \setminus (E \cap \overline Q_1)$ and find $1-v \ge \varphi$.
On the other hand since $\varphi(0,1) \ge \delta$ and $\varphi =0$ on the lateral boundary of $Q_2(1,0) \times [0,1]$ it follows that $\varphi$ satisfies the forward Harnack inequality, and $\varphi \ge c \delta$ in $Q_{1/2}(0,1)$. The same inequality holds for $1-v$ which gives the desired estimate.

\end{proof}

\begin{rem}\label{R41}
We may write the conclusion in $[-1/2,1/2]^n \times [\sigma,1]$ for any $\sigma >0$ provided that the constant $c=c(\delta, \sigma)$ depends on $\sigma$ as well.
\end{rem} 

We are now ready to provide the proof of Theorem \ref{T4}.

\

{\it Proof of Theorem $\ref{T4}$.}
We only show that the exponent in the estimate \eqref{step2} from the previous section can be improved to
\begin{equation}\label{++}
w^- \le e^{-c r^{1- \frac \beta \alpha}},
\end{equation}
by the use of the two lemmas above. The rest of the proof remains the same as before. Notice that now $1- \frac \beta \alpha < \gamma $ holds simply by choosing $\beta >1$ and no restriction on range of the H\"older exponent $\alpha >0$ is needed. 

The same argument improves the exponent in \eqref{step22} from $1-\beta$ to $1- \beta/\alpha$ in the proof of the Carleson estimate.  

We proceed with the proof of \eqref{++}. We set $\bar r:=  r^{\beta/\alpha}$, and 
by hypothesis, the translation by the vector $$T:=(-\bar r e_n, \kappa \,  \bar r ^2) \, \, \in \R^{n+1}$$ maps the complement of $\mathcal C_1$ into itself, provided that $\kappa \le 1/2$ is small depending on the constant $K$ in \eqref{lower}.
Thus if we take a cube and then translate it by $T$, the complement of $\mathcal C_1$ (where $w^-=0$) ``increased" in the translating cube because of \eqref{lower} (see Figure 5).

Decompose the space $\R^{n+1}$ into cubes of size $Q_{\bar r}$ in the following way. Take $Q_{\bar r}$ centered at the origin and then translate it by a 
linear combination of the vectors $\bar r e_i$, $i<n$, $T$ and $\bar r^2 e_{n+1}$ using integer coefficients. 
We look at the behavior of $w$ on arrays of cubes translated by multiples of $T$. Starting with $Q_{\bar r}(0)$, we consider $Q_{\bar r}(mT)$, with $m \in \mathbb Z$. 
When $m \ge C \bar r ^{\alpha -1}$, $Q_{3\bar r }(mT) \subset \mathcal A_r$, and when $m \le - C \bar r ^{\alpha -1}$, $Q_{\bar r}(m T) \subset E$, where $E$ denotes the complement of $\mathcal C_1$. 
Thus, there is an intermediate $m_0$ where $$cap_{Q_{3 \bar r}(mT)}(E)<\delta \quad \mbox{ if and only if $m \ge m_0$.}$$ 
When we decrease $m$ from $C \bar r ^{\alpha -1}$ to $m_0$ we may apply Lemma \ref{l1} in each such $Q_{3 \bar r}(mT)$. The weak Harnack inequality holds in measure in these cubes (see Remark \ref{R40}, with $ \sigma^2 = \kappa /20$), and as in \eqref{step1}, (as there are at most $C \bar r ^{\alpha -1}$ such cubes) we find that 
$$\{w +1 \ge f(r) e^{- C \bar r^{\alpha-1}} \ge 2\}$$ in a fixed proportion of each such $Q_{3\bar r}(mT)$ with $m \ge m_0$. Thus $w^-=0$ in a fixed proportion of $Q_{3 \bar r}(mT)$, and by the weak Harnack inequality
\begin{equation} \label{4000}
\|w^-\|_{L^\infty (Q_{\bar r}(x,t)) } \le (1-c) \|w^-\|_{L^\infty (Q_{6 \bar r}(x,t))} \quad \quad (x,t)=mT + 2 \bar r^2 e_{n+1},
\end{equation}
if $m \ge m_0$.

If $m < m_0$ then the capacity of $E$ in $Q_{3 \bar r}(m T)$ is more than $\delta$. By 
Lemma \ref{l2}, the inequality above remains valid after possibly relabeling $c$. We conclude that 
\eqref{4000} is valid for all cubes centered at $mT + 2 \bar r^2 e_{n+1}$, and in particular for $Q_{\bar r}(2\bar r^2 e_{n+1})$. 

This argument shows that \eqref{4000} holds in fact at all  points $(x,t) \in \mathcal C^-_{3r/4}$. Indeed,
$$ \text{if $|x_n - g(x',t)| > C r^\beta$ then either $Q_{\bar r}
(x,t) \subset \mathcal A_r$ or $Q_{\bar r}(x,t) \subset E$}$$and $\eqref{4000}$ is satisfied trivially as $w^-=0$ in $Q_{\bar r}(x,t)$. 
Otherwise, we argue as above by decomposing the space starting with the cube centered at $(x,t) - 2 \bar r^2 e_{n+1}$ instead of the origin. 
Notice that  $(x,t) \in \mathcal C^-_{3r/4}$ and $|x_n - g(x',t)| \le  C r^\beta$ imply that  $$\text{$Q_{3\bar r }((x,t)-mT) \subset \mathcal A_r$ and $Q_{3\bar r }
((x,t) + mT) \subset E$ when $m \sim C \bar r^{\alpha -1}$ }$$
and the argument applies as before.

In conclusion, the maximum of $w^-$ is decaying a fixed proportion each time we remove the cubes $Q_{6\bar r} $ which are tangent to the parabolic boundary of the infinite cylinder in the $(x',t)$ variables $$\{|x_i| \le 3/4 r, \quad i < n\} \cap \{t \in [-3/4 \, r,0]\}.$$
Thus $$w^- \le e^{-cr\bar r^{-1}} \quad \mbox{in $\mathcal C_{r/2}^-$,}$$
as desired, and \eqref{++} is proved.

\qed

\begin{figure}[h] 
\includegraphics[width=0.5 \textwidth]{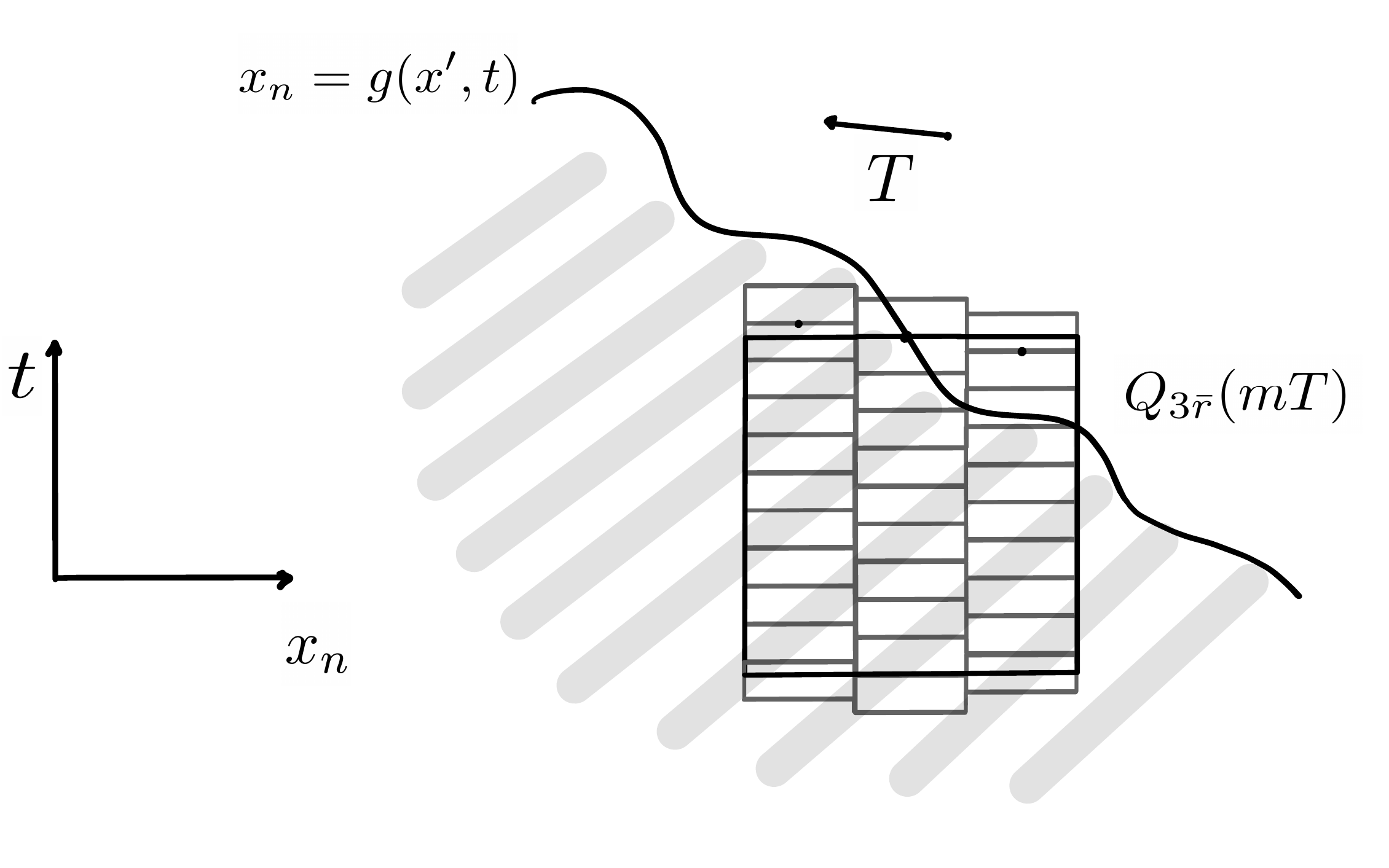}
\caption{ Proof of Theorem \ref{T4}.}
   \label{fig5}
\end{figure}

\end{document}